 \documentclass[final,3p,times]{elsarticle}

\usepackage{amssymb}
\usepackage{amsthm}

\journal{ }

\usepackage[all]{xy}
\usepackage{amsmath,amscd}
\newtheorem{lem}{Lemma}[section]
\newtheorem{prop}[lem]{Proposition}
\newtheorem{cor}[lem]{Corollary}
\newtheorem{thm}[lem]{Theorem}
\theoremstyle{definition}
\newtheorem{defin}[lem]{Definition}
\newtheorem{remark}[lem]{Remark}

\begin{document}

\begin{frontmatter}



\title{Hochschild cohomology for a class of some 
self-injective special biserial algebras of rank four}


\author{Takahiko Furuya\footnote{E-mail address: 
furuya@dent.meikai.ac.jp}}

\address{School of Dentistry, Meikai University, 
1-1 Keyakidai, Sakado, Saitama, Japan}

\begin{abstract}
In this paper, we construct an explicit minimal projective 
bimodule resolution of a self-injective special biserial 
algebra $A_{T}$ ($T\geq0$) whose Grothendieck group is of rank 
$4$. As a main result, we determine the dimension of the 
Hochschild cohomology group ${\rm HH}^{i}(A_{T})$ of $A_{T}$ 
for $i\geq0$, completely. Moreover we give a presentation of 
the Hochschild cohomology ring modulo nilpotence ${\rm HH}^{*}
(A_{T})/\mathcal{N}_{A_{T}}$ of $A_{T}$ by generators and 
relations in the case where $T=0$. 
\end{abstract}


\end{frontmatter}


\section{Introduction}
\label{introduction}

Let $\Gamma$ be the following circular quiver with four vertices 
$0$, $1$, $2$, $3$ and eight arrows $a_{i}$, $b_{i}$ for 
$i=0,1,2,3$: 
$$
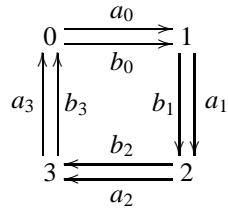
\begin{xy}
(0,0)		*+{3}="e3",
(18,0)		*+{2}="e2",
(18,18)		*+{1}="e1",
(0,18)		*+{0}="e0",
(9.5,21)	*{a_{0}},
(9.5,14.5)	*{b_{0}},
(22,9)		*{a_{1}},
(15,9)		*{b_{1}},
(9.5,-3.5)	*{a_{2}},
(9.5,3.5)	*{b_{2}},
(-3.5,9)	*{a_{3}},
(3.5,9)		*{b_{3}},
\ar @<1mm> "e0"; "e1",
\ar @<-1mm> "e0"; "e1",
\ar @<1mm> "e1"; "e2",
\ar @<-1mm> "e1"; "e2",
\ar @<1mm> "e2"; "e3",
\ar @<-1mm> "e2"; "e3",
\ar @<1mm> "e3"; "e0",
\ar @<-1mm> "e3"; "e0",
\end{xy}$$
Denote the trivial path corresponding to the vertex $i$ by 
$e_{i}$ for $0\leq i\leq3$. We always consider the subscripts 
$i$ of $e_{i}$, $a_{i}$ and $b_{i}$ as modulo $4$. Therefore 
the arrows $a_{i}$ and $b_{i}$ start at $e_{i}$ and end with 
$e_{i+1}$ for all $i\in\mathbb{Z}$. Paths are written from left 
to right.

Let $K$ be an algebraically closed field, and let $K\Gamma$ be 
the path algebra of $\Gamma$ over $K$. We set $x:=\sum_{i=0}^{3}
a_{i}\in K\Gamma$ and $y:=\sum_{i=0}^{3}b_{i}\in K\Gamma$. Then, 
for integers $0\leq i\leq3$ and $j\geq0$, the elements $e_{i}x^{j}$ 
and $e_{i}y^{j}$ are precisely the paths $a_{i}a_{i+1}\cdots 
a_{i+j-1}$ and $b_{i}b_{i+1}\cdots b_{i+j-1}$ of length 
$j$ respectively, so that $e_{i}x^{j}=e_{i}x^{j}e_{i+j}=x^{j}
e_{i+j}$ and $e_{i}y^{j}=e_{i}y^{j}e_{i+j}=y^{j}e_{i+j}$ hold. 
Fix an integer $T\geq0$, and let $I_{T}$ be the ideal in $K\Gamma$ 
generated by the elements $xy$, $x^{4T+2}+y^{4T+2}$ and $yx$, 
that is, $I_{T}:=\langle xy,\ x^{4T+2}+y^{4T+2},\ yx\rangle$. 
Define the algebra $A_{T}$ to be the quotient $K\Gamma/I_{T}$. 
We then see that the set $\{e_{i}x^{j},\:e_{i}y^{l}\mid0\leq 
i\leq3;\ 0\leq j\leq4T+2;\ 1\leq l\leq 4T+1\}$ is a $K$-basis 
of $A_{T}$, so that $\dim_{K}A_{T}=16(2T+1)$. Furthermore 
$A_{T}$ is a self-injective special biserial algebra, and 
hence is of tame representation type. In 
particular, if $T=0$, then we see that $A_{0}$ is a Koszul 
algebra of radical cube zero (see Proposition~\ref{koszul}).

The purpose in this paper is to investigate the Hochschild 
cohomology for $A_{T}$. The Hochschild cohomology groups and 
rings of algebras are important invariants in the representation 
theory of algebras, and have been studied by many researchers. 
However, in general, it is not easy to describe their structures, 
even if the given algebras are easier to deal with.

Recently, in the papers \cite{ES,PS,SS1,ST}, the Hochschild 
cohomology groups or rings of certain finite-dimensional 
self-injective algebras were described, where the authors 
provided projective bimodule resolutions by using certain 
sets $\mathcal{G}^{n}$ found in \cite{GSZ}. These sets are also 
used in the papers \cite{F,GHMS,GS} in constructing projective 
bimodule resolutions. In this paper, following this technique, 
we give a projective bimodule resolution of $A_{T}$ for all 
$T\geq0$, and then study its Hochschild cohomology.

One important problem in the study of Hochschild cohomology 
is to find necessarily and sufficient conditions for the 
Hochschild cohomology ring modulo nilpotence to be finitely 
generated as an algebra. So far it has been proved that the 
Hochschild cohomology rings modulo nilpotence for several 
classes of finite-dimensional algebras, such as group algebras 
\cite{E,V}, self-injective algebras of finite representation 
type \cite{GSS1}, monomial algebras \cite{GSS2}, are finitely 
generated. Also, some examples of infinitely generated 
Hochschild cohomology rings modulo nilpotence can be found 
in \cite{S, X}. However, a definitive answer to this problem 
has not yet been obtained. In this paper, we show that the 
Hochschild cohomology ring modulo nilpotence of $A_{T}$ is 
finitely generated in the case where $T=0$.

This paper is organized as follows: In 
Section~\ref{section_resolution}, we construct sets 
$\mathcal{G}^{n}$ $(n\geq0)$ for the right $A_{T}$-module 
$A_{T}/\mathfrak{r}_{A_{T}}$, and then give an explicit 
minimal projective bimodule resolution of $A_{T}$. In 
Section~\ref{section_group}, we find a $K$-basis of the 
Hochschild cohomology group ${\rm HH}^{i}(A_{T})$ $(i\geq0)$ 
of $A_{T}$ for all $T\geq0$, and then determine its dimension, 
completely. In Section~\ref{modulo_nilpotence}, we give a 
presentation by generators and relations of the Hochschild 
cohomology ring modulo nilpotence ${\rm HH}^{*}(A_{T})/
\mathcal{N}_{A_{T}}$ in the Koszul case, $T=0$, and show that 
${\rm HH}^{*}(A_{0})/\mathcal{N}_{A_{0}}$ is finitely generated.

For every arrow $c$ in $\Gamma$, we denote the origin by 
$\mathfrak{o}(c)$ and the terminus by $\mathfrak{t}(c)$. For 
simplicity we write $\otimes_{K}$ as $\otimes$. Moreover we 
denote the enveloping algebra $A_{T}^{\rm op}\otimes A_{T}$ 
of $A_{T}$ by $A_{T}^{\rm e}$. Note that there is a natural 
one to one correspondence between the family of 
$A_{T}$-$A_{T}$-bimodules and that of right 
$A_{T}^{\rm e}$-modules. We also denote the Jacobson radical 
of $A_{T}$ by $\mathfrak{r}_{A_{T}}$.

\section{The sets $\mathcal{G}^{n}$ and a projective bimodule 
resolution $(Q^{\bullet},\partial)$ of $A_{T}$}
\label{section_resolution}

Let $\Lambda=K\mathcal{Q}/I$ be any finite-dimensional $K$-algebra 
with $\mathcal{Q}$ a finite quiver, and $I$ an admissible ideal 
in $K\mathcal{Q}$, and let $\mathfrak{r}_{\Lambda}$ be the 
Jacobson radical of $\Lambda$. We denote by $\mathcal{G}^{0}$ 
the set of all vertices of $\mathcal{Q}$, by $\mathcal{G}^{1}$ 
the set of all arrows of $\mathcal{Q}$, and by $\mathcal{G}^{2}$ 
a minimal set of uniform generators of $I$. In \cite{GSZ}, Green, 
Solberg and Zacharia showed that, for each $n\geq3$, there is 
a set $\mathcal{G}^{n}$ of uniform elements in $K\mathcal{Q}$ 
such that we have a minimal projective resolution $(P^{\bullet},d)$ 
of the right $\Lambda$-module $\Lambda/\mathfrak{r}_{\Lambda}$ 
satisfying the following conditions: 
\begin{enumerate}[(a)]
\item For $n\geq0$, $P^{n}=\bigoplus_{x\in\mathcal{G}^{n}}
\mathfrak{t}(x)\Lambda$.

\item For $x\in\mathcal{G}^{n}$, there are unique elements 
$r_{y},s_{z}\in K\mathcal{Q}$, where $y\in\mathcal{G}^{n-1}$ 
and $z\in\mathcal{G}^{n-2}$, such that $x=\sum_{y\in
\mathcal{G}^{n-1}}yr_{y}=\sum_{z\in\mathcal{G}^{n-2}}zs_{z}$.

\item For $n\geq1$, the differential $d^{n}:P^{n}\rightarrow 
P^{n-1}$ is defined by $d^{n}(\mathfrak{t}(x)\lambda):=
\sum_{y\in\mathcal{G}^{n-1}}r_{y}\mathfrak{t}(x)\lambda$ for 
$x\in\mathcal{G}^{n}$ and $\lambda\in\Lambda$, where $r_{y}$ 
denotes the element in the expression (b).
\end{enumerate}
In this section, we will construct sets $\mathcal{G}^{n}$ 
$(n\geq0)$ for the right $A_{T}$-module $A_{T}/
\mathfrak{r}_{A_{T}}$, and then use them to give a projective 
bimodule resolution $(Q^{\bullet},\partial)$ of $A_{T}$.

\subsection{Sets $\mathcal{G}^{n}$ for $A_{T}/
\mathfrak{r}_{A_{T}}$}
First, in order to give sets $\mathcal{G}^{n}$ ($n\geq0$) 
for $A_{T}/\mathfrak{r}_{A_{T}}$, we introduce the following 
elements in $K\Gamma$:

\begin{defin}
For $0\leq i\leq3$, we put $g_{i,0}^{0}:=e_{i}$. Furthermore, 
for $n\geq1$, we inductively define the elements $g_{i,j}^{n}
\in K\Gamma$ for $0\leq i\leq3$ and $0\leq j\leq n$ as follows: 
\begin{enumerate}[(a)]
\item If $n=2m+1$ with $m\geq0$, then  
$$g_{i,j}^{2m+1}:=\begin{cases}
g_{i,0}^{2m}x 	& \mbox{if $i=0,2$ and $j=0$}\\ 
g_{i,0}^{2m}y 	& \mbox{if $i=1,3$ and $j=0$}\\ 
g_{i,j-1}^{2m}y^{4T+1}+g_{i,j}^{2m}x 
		& \mbox{if $i=0,2$ and $1\leq j\leq m$}\\
g_{i,j-1}^{2m}x^{4T+1}+g_{i,j}^{2m}y 
		& \mbox{if $i=1,3$ and $1\leq j\leq m$}\\
g_{i,j-1}^{2m}y+g_{i,j}^{2m}x^{4T+1}
		& \mbox{if $i=0,2$ and $m+1\leq j\leq2m$}\\
g_{i,j-1}^{2m}x+g_{i,j}^{2m}y^{4T+1}
		& \mbox{if $i=1,3$ and $m+1\leq j\leq2m$}\\
g_{i,2m}^{2m}y 	& \mbox{if $i=0,2$ and $j=2m+1$}\\
g_{i,2m}^{2m}x 	& \mbox{if $i=1,3$ and $j=2m+1$.}
\end{cases}$$
\item If $n=2m$ with $m\geq1$, then  
$$g_{i,j}^{2m}:=\begin{cases}
g_{i,0}^{2m-1}y & \mbox{if $i=0,2$ and $j=0$}\\
g_{i,0}^{2m-1}x & \mbox{if $i=1,3$ and $j=0$}\\
g_{i,j-1}^{2m-1}x^{4T+1}+g_{i,j}^{2m-1}y 
		& \mbox{if $i=0,2$ and $1\leq j\leq m-1$}\\
g_{i,j-1}^{2m-1}y^{4T+1}+g_{i,j}^{2m-1}x 
		& \mbox{if $i=1,3$ and $1\leq j\leq m-1$}\\
g_{i,m-1}^{2m-1}x^{4T+1}+g_{i,m}^{2m-1}y^{4T+1} 
		& \mbox{if $i=0,2$ and $j=m$}\\
g_{i,m-1}^{2m-1}y^{4T+1}+g_{i,m}^{2m-1}x^{4T+1} 
		& \mbox{if $i=1,3$ and $j=m$}\\
g_{i,j-1}^{2m-1}x+g_{i,j}^{2m-1}y^{4T+1}
		& \mbox{if $i=0,2$ and $m+1\leq j\leq 2m-1$}\\
g_{i,j-1}^{2m-1}y+g_{i,j}^{2m-1}x^{4T+1}
		& \mbox{if $i=1,3$ and $m+1\leq j\leq 2m-1$}\\
g_{i,2m-1}^{2m-1}x& \mbox{if $i=0,2$ and $j=2m$}\\
g_{i,2m-1}^{2m-1}y& \mbox{if $i=1,3$ and $j=2m$.}
\end{cases}$$

\end{enumerate}
\end{defin}
\noindent 
We immediately see that these elements $g_{i,j}^{n}$ are 
uniform.

Now we put the set 
$$\mathcal{G}^{n}:=\left\{g^{n}_{i,j}\,\Big|\,0\leq i\leq3;\ 
0\leq j\leq n\right\}$$
for all $n\geq0$.

\begin{remark}\label{rem1}
\begin{enumerate}[(a)]
\item For all $n\geq0$, $0\leq i\leq3$ and $0\leq j\leq n$, 
we have $\mathfrak{o}(g_{i,j}^{n})=e_{i}$. Also, if $i+n\equiv t$ 
(${\rm mod}\,4$), then $\mathfrak{t}(g_{i,j}^{n})=e_{t}$. 
\item We get 
\begin{align*}
\mathcal{G}^{0}&=\{e_{i}\mid0\leq i\leq3\},\\
\mathcal{G}^{1}&=\{a_{i},\:b_{i}\mid0\leq i\leq3\},\\
\mathcal{G}^{2}&=\left\{e_{i}xy,\:e_{i}(x^{4T+2}+y^{4T+2}),
\:e_{i}yx\ \big|\ 0\leq i\leq3\right\}, 
\end{align*}
and so $\mathcal{G}^{2}$ is a minimal set of generators of $I_{T}$. 
\end{enumerate}
\end{remark}

\noindent
It is not hard to check that these sets satisfy the conditions 
(a), (b), and (c) in the beginning of this section.

Now it can be seen that, for all $T\geq0$, $A_{T}$ is a 
self-injective algebra. Moreover, if $T=0$, then we have 
the following proposition.

\begin{prop}\label{koszul}
The algebra $A_{0}$ is a self-injective Koszul algebra. 
\end{prop}

\begin{proof}
If $T=0$, then we notice that the resolution $(P^{\bullet},d)$ 
determined by (a), (b) and (c) above is a linear resolution of 
$A_{0}/\mathfrak{r}_{A_{0}}$, and hence $A_{0}$ is a Koszul 
self-injective algebra. 
\end{proof}

\subsection{A projective bimodule resolution of $A_{T}$}
Now we give a projective bimodule resolution $(Q^{\bullet}, 
\partial)$ for $A_{T}$. For simplicity we denote the element 
$\mathfrak{o}(g_{i,j}^{n})\otimes\mathfrak{t}(g_{i,j}^{n})$ 
in $A_{T}\mathfrak{o}(g_{i,j}^{n})\otimes\mathfrak{t}
(g_{i,j}^{n})A_{T}$ by $\mathfrak{a}_{i,j}^{n}$ for $n\geq0$, 
$0\leq i\leq3$, and $0\leq j\leq n$. For $n\geq0$, we define 
the projective $A_{T}$-$A_{T}$-bimodule $Q^{n}$ by 
$$Q^{n}:=\bigoplus_{g\in\mathcal{G}^{n}}A_{T}\mathfrak{o}(g)
\otimes\mathfrak{t}(g)A_{T}=\bigoplus_{i=0}^{3}
\Bigg[\bigoplus_{j=0}^{n}A_{T}\mathfrak{a}_{i,j}^{n}A_{T}
\Bigg].$$
Furthermore we define the map $\partial^{n}$ as follows: 
\begin{defin}
Define $\partial^{0}:Q^{0}\rightarrow A_{T}$ to be the multiplication 
map, and for $n\geq1$ define $\partial^{n}:Q^{n}\rightarrow Q^{n-1}$ 
to be the homomorphism of $A_{T}$-$A_{T}$-bimodules determined 
by the following formulas. Here, we note that the lower left 
subscripts $i$ of $\mathfrak{a}_{i,j}^n$ are considered as 
modulo $4$. 
\begin{enumerate}[(a)]

\item If $n=2m+1$ for $m\geq0$, then, for $0\leq i\leq3$ and 
$0\leq j\leq 2m+1$, 
$$\partial^{2m+1}(\mathfrak{a}_{i,j}^{2m+1}):=
\begin{cases}
\mathfrak{a}_{i,0}^{2m}x-x\mathfrak{a}_{i+1,0}^{2m} 
		& \mbox{if $i=0,2$ and $j=0$}\\
\mathfrak{a}_{i,0}^{2m}y-y\mathfrak{a}_{i+1,0}^{2m} 
		& \mbox{if $i=1,3$ and $j=0$}\\
\mathfrak{a}_{i,j-1}^{2m}y^{4T+1}+\mathfrak{a}_{i,j}^{2m}x-x
\mathfrak{a}_{i+1,j}^{2m}-y^{4T+1}\mathfrak{a}_{i+1,j-1}^{2m} 
		& \mbox{if $i=0,2$ and $1\leq j\leq m$}\\
\mathfrak{a}_{i,j-1}^{2m}x^{4T+1}+\mathfrak{a}_{i,j}^{2m}y-y
\mathfrak{a}_{i+1,j}^{2m}-x^{4T+1}\mathfrak{a}_{i+1,j-1}^{2m} 
		& \mbox{if $i=1,3$ and $1\leq j\leq m$}\\
\mathfrak{a}_{i,j-1}^{2m}y+\mathfrak{a}_{i,j}^{2m}x^{4T+1}-x^{4T+1}
\mathfrak{a}_{i+1,j}^{2m}-y\mathfrak{a}_{i+1,j-1}^{2m} 
		& \mbox{if $i=0,2$ and $m+1\leq j\leq 2m$}\\
\mathfrak{a}_{i,j-1}^{2m}x+\mathfrak{a}_{i,j}^{2m}y^{4T+1}-y^{4T+1}
\mathfrak{a}_{i+1,j}^{2m}-x\mathfrak{a}_{i+1,j-1}^{2m} 
		& \mbox{if $i=1,3$ and $m+1\leq j\leq2m$}\\
\mathfrak{a}_{i,2m}^{2m}y-y\mathfrak{a}_{i+1,2m}^{2m}
		& \mbox{if $i=0,2$ and $j=2m+1$}\\
\mathfrak{a}_{i,2m}^{2m}x-x\mathfrak{a}_{i+1,2m}^{2m}
		& \mbox{if $i=1,3$ and $j=2m+1$.} 
\end{cases}$$

\item If $n=2m$ for $m\geq1$, then, for $0\leq i\leq3$ and 
$0\leq j\leq2m$, 
$$\partial^{2m}(\mathfrak{a}_{i,j}^{2m}):=\begin{cases}
\mathfrak{a}_{i,0}^{2m-1}y+x\mathfrak{a}_{i+1,0}^{2m-1} 
	& \mbox{if $i=0,2$ and $j=0$}\\
\mathfrak{a}_{i,0}^{2m-1}x+y\mathfrak{a}_{i+1,0}^{2m-1} 
	& \mbox{if $i=1,3$ and $j=0$}\\
\mathfrak{a}_{i,j-1}^{2m-1}x^{4T+1}+\mathfrak{a}_{i,j}^{2m-1}
y+x\mathfrak{a}_{i+1,j}^{2m-1}+y^{4T+1}\mathfrak{a}_{i+1,j-1}^{2m-1} 
	& \mbox{if $i=0,2$ and $1\leq j\leq m-1$}\\
\mathfrak{a}_{i,j-1}^{2m-1}y^{4T+1}+\mathfrak{a}_{i,j}^{2m-1}
x+y\mathfrak{a}_{i+1,j}^{2m-1}+x^{4T+1}\mathfrak{a}_{i+1,j-1}^{2m-1} 
	& \mbox{if $i=1,3$ and $1\leq j\leq m-1$}\\
\Big[\sum_{s=0}^{T}x^{4s}\big(\mathfrak{a}_{i,m-1}^{2m-1}x
+x\mathfrak{a}_{i+1,m}^{2m-1}\big)x^{4T-4s}\Big] 
	& \\
\quad+\Big[\sum_{s=0}^{T-1}x^{4s+2}\big(\mathfrak{a}_{i+2,m-1}^{2m-1}
x+x\mathfrak{a}_{i+3,m}^{2m-1}\big)x^{4T-4s-2}\Big] \\
\quad\quad+\Big[\sum_{s=0}^{T}y^{4s}\big(\mathfrak{a}_{i,m}^{2m-1}
y+y\mathfrak{a}_{i+1,m-1}^{2m-1}\big)y^{4T-4s}\Big]\\
\quad\quad\quad+\Big[\sum_{s=0}^{T-1}y^{4s+2}
\big(\mathfrak{a}_{i+2,m}^{2m-1}y+y\mathfrak{a}_{i+3,m-1}^{2m-1}\big)
y^{4T-4s-2}\Big]
	& \mbox{if $i=0,2$ and $j=m$}\\
\Big[\sum_{s=0}^{T}y^{4s}\big(\mathfrak{a}_{i,m-1}^{2m-1}y+y
\mathfrak{a}_{i+1,m}^{2m-1}\big)y^{4T-4s}\Big]\\
\quad+\Big[\sum_{s=0}^{T-1}y^{4s+2}\big(\mathfrak{a}_{i+2,m-1}^{2m-1}
y+y\mathfrak{a}_{i+3,m}^{2m-1}\big)y^{4T-4s-2}\Big]\\
\quad\quad+\Big[\sum_{s=0}^{T}x^{4s}\big(\mathfrak{a}_{i,m}^{2m-1}
x+x\mathfrak{a}_{i+1,m-1}^{2m-1}\big)x^{4T-4s}\Big]\\
\quad\quad\quad+\Big[\sum_{s=0}^{T-1}x^{4s+2}\big(\mathfrak{a}_{i+2,
m}^{2m-1}x+x\mathfrak{a}_{i+3,m-1}^{2m-1}\big)x^{4T-4s-2}\Big]
	&\mbox{if $i=1,3$ and $j=m$}\\
\mathfrak{a}_{i,j-1}^{2m-1}x+\mathfrak{a}_{i,j}^{2m-1}y^{4T+1}+
x^{4T+1}\mathfrak{a}_{i+1,j}^{2m-1}+y\mathfrak{a}_{i+1,j-1}^{2m-1} 
	& \mbox{if $i=0,2$ and $m+1\leq j\leq 2m-1$}\\
\mathfrak{a}_{i,j-1}^{2m-1}y+\mathfrak{a}_{i,j}^{2m-1}x^{4T+1}+
y^{4T+1}\mathfrak{a}_{i+1,j}^{2m-1}+x\mathfrak{a}_{i+1,j-1}^{2m-1} 
	& \mbox{if $i=1,3$ and $m+1\leq j\leq 2m-1$}\\
\mathfrak{a}_{i,2m-1}^{2m-1}x+y\mathfrak{a}_{i+1,2m-1}^{2m-1}
	& \mbox{if $i=0,2$ and $j=2m$}\\
\mathfrak{a}_{i,2m-1}^{2m-1}y+x\mathfrak{a}_{i+1,2m-1}^{2m-1}
	& \mbox{if $i=1,3$ and $j=2m$}. 
\end{cases}$$

\end{enumerate}
\end{defin}
\noindent
It is straightforward to check that the composite 
$\partial^{n}\partial^{n+1}$ is zero for all $n\geq0$, so that 
$(Q^{\bullet},\partial)$ is a complex of 
$A_{T}$-$A_{T}$-bimodules.

\begin{remark}\label{resolution_A_T}
For $n\geq0$, the map $G^{n}:A_{T}/\mathfrak{r}_{A_{T}}\otimes_{A_{T}}
Q^{n}\rightarrow P^{n}$ determined by $G^{n}\big(\mathfrak{o}
(g_{i,j}^{n})\otimes_{A_{T}}\mathfrak{a}_{i,j}^{n}\big)=\mathfrak{t}
(g_{i,j}^{n})$ ($0\leq i\leq3$; $0\leq j\leq n$) is an isomorphism 
of right $A_{T}$-modules, and this map makes the following 
diagram commutative: 
$$\begin{CD}
A_{T}/\mathfrak{r}_{A_{T}}\otimes_{A_{T}}Q^{n+1}&@>A_{T}/
\mathfrak{r}_{A_{T}}\otimes_{A_{T}}\partial^{n+1}>> & 
A_{T}/\mathfrak{r}_{A_{T}}\otimes_{A_{T}}Q^{n} \\
@VG^{n+1}V{\simeq}V&& @V{\simeq}VG^{n}V \\
P^{n+1} & @>d^{n+1}>>&P^{n}
\end{CD}$$
This shows that $(A_{T}/\mathfrak{r}_{A_{T}}\otimes_{A_{T}}
Q^{\bullet},A_{T}/\mathfrak{r}_{A_{T}}\otimes_{A_{T}}\partial)$ 
is isomorphic to $(P^{\bullet},d)$ as complexes and hence is 
a minimal projective resolution of $A_{T}/\mathfrak{r}_{A_{T}}
\otimes_{A_{T}}A_{T}$ ($\simeq A_{T}/\mathfrak{r}_{A_{T}}$). 
\end{remark}

\noindent
Now we have the following theorem. The proof is done with 
Remark~\ref{resolution_A_T} and by following 
\cite{GS} (and see also \cite{ST}), so we omit it.

\begin{thm}
The complex $(Q^{\bullet},\partial)$ is a minimal projective 
bimodule resolution of $A_{T}$. 
\end{thm}

\section{Hochschild cohomology groups of $A_{T}$}
\label{section_group}
In this section we find an explicit $K$-basis of the 
Hochschild cohomology group ${\rm HH}^{i}(A_{T})$ $(i\geq0)$ by 
using the resolution $(Q^{\bullet},\partial)$ in 
Section~\ref{section_resolution}, and then give the 
dimension of ${\rm HH}^{i}(A_{T})$, completely. Throughout this 
section we keep the notation from Section~\ref{section_resolution}.

By applying the functor ${\rm Hom}_{A_{T}^{\rm e}}(-,A_{T})$ to 
$(Q^{\bullet},\partial)$, we have the complex 
$$\begin{CD}
0 \longrightarrow\ 	& {\rm Hom}_{A_{T}^{\rm e}}(Q^{0},A_{T}) 
	& @>{\rm Hom}_{A_{T}^{\rm e}}(\partial^{1},A_{T})>> 
	& {\rm Hom}_{A_{T}^{\rm e}}(Q^{1},A_{T})
	& @>{\rm Hom}_{A_{T}^{\rm e}}(\partial^{2},A_{T})>> 
	&		& \bigskip\\ 
	& {\rm Hom}_{A_{T}^{\rm e}}(Q^{2},A_{T}) 
	& @>{\rm Hom}_{A_{T}^{\rm e}}(\partial^3,A_{T})>> 
	& {\rm Hom}_{A_{T}^{\rm e}}(Q^{3},A_{T})
	& @>{\rm Hom}_{A_{T}^{\rm e}}(\partial^{4},A_{T})>> 
	& \cdots. \end{CD}$$
Recall that, for $n\geq0$, the $n$th Hochschild cohomology 
group ${\rm HH}^{n}(A_{T})$ of $A_{T}$ is defined to be the 
$K$-space ${\rm HH}^{n}(A_{T}):={\rm Ext}^{n}_{A_{T}^{\rm e}}(A_{T}, 
A_{T})={\rm Ker}\,{\rm Hom}_{A_{T}^{\rm e}}(\partial^{n+1},A_{T})
/{\rm Im}\,{\rm Hom}_{A_{T}^{\rm e}}(\partial^{n},A_{T})$.

\subsection{A basis of ${\rm Hom}_{A_{T}^{\rm e}}(Q^{i},A_{T})$}
We start with the following remark: 
\begin{remark}\label{rembimodule}
\begin{enumerate}[\rm(a)]
\item For integers $n\geq0$, $0\leq i\leq3$ and $0\leq j\leq n$, if 
$n\equiv t$ (${\rm mod}\,4$), then, by Remark~\ref{rem1}~(a), we 
get $\mathfrak{o}(g_{i,j}^{n})A_{T}\mathfrak{t}(g_{i,j}^{n})=
e_{i}A_{T}e_{i+t}$. Hence $\mathfrak{o}(g_{i,j}^{n})A_{T}\mathfrak{t}
(g_{i,j}^{n})$ has a $K$-basis 
$$\begin{cases}
\{e_{i}\} 	& \mbox{if $T=0$ and $n\equiv0$ (${\rm mod}\,4$)}\\ 
\{e_{i}x^{4l},\:e_{i}y^{4u}\mid0\leq l\leq T;\ 1\leq u\leq T\} 
		& \mbox{if $T\geq1$ and $n\equiv0$ (${\rm mod}\,4$)}\\ 
\{e_{i}x^{4l+1},\:e_{i}y^{4l+1}\mid0\leq l\leq T\} 
		& \mbox{if $n\equiv1$ (${\rm mod}\,4$)}\\
\{e_{i}x^{2}\}
		& \mbox{if $T=0$ and $n\equiv2$ (${\rm mod}\,4$)}\\
\{e_{i}x^{4l+2},\:e_{i}y^{4u+2}\mid0\leq l\leq T;\ 0\leq u\leq T-1\} 
		& \mbox{if $T\geq1$ and $n\equiv2$ (${\rm mod}\,4$)}\\
\{e_{i}x^{4l+3},\:e_{i}y^{4l+3}\mid0\leq l\leq T-1\} 
		& \mbox{if $T\geq1$ and $n\equiv3$ (${\rm mod}\,4$).}
\end{cases}$$
Moreover, $\mathfrak{o}(g_{i,j}^{n})A_{T}\mathfrak{t}(g_{i,j}^{n})
=\{0\}$, if $T=0$ and $n\equiv3$ (${\rm mod}\,4$).

\item For $n\geq0$ the map $F:\bigoplus_{g\in\mathcal{G}^{n}}
\mathfrak{o}(g)A_{T}\mathfrak{t}(g)\rightarrow
{\rm Hom}_{A_{T}^{\rm e}}(Q^{n},A_{T})$ given by 
$(F(\sum_{g\in\mathcal{G}^{n}}z_{g}))(\mathfrak{a}_{i,j}^{n})=
z_{g_{i,j}^{n}}$, where $z_{g}\in\mathfrak{o}(g)A_{T}\mathfrak{t}
(g)$ for $g\in \mathcal{G}^{n}$, $0\leq i\leq3$ and $0\leq j
\leq n$, is an isomorphism of $K$-spaces. 
\end{enumerate}
\end{remark}
\noindent 
We need the following maps. 
\begin{defin} Let $n\geq0$ be an integer. For $0\leq i
\leq3$, $0\leq j\leq n$ and $\begin{cases}0\leq l\leq T&
\mbox{if $n\not\equiv3$ (${\rm mod}\,4$)}\\ 0\leq l\leq T-1&
\mbox{if $n\equiv3$ (${\rm mod}\,4$),}\end{cases}$ we define 
the maps $\beta_{i,j}^{n,l}$, $\gamma_{i,j}^{n,l}$: $Q^{n}
\rightarrow A_{T}$ to be the homomorphisms of 
$A_{T}$-$A_{T}$-bimodules determined by 
$$\beta_{i,j}^{n,l}(\mathfrak{a}_{r,s}^{n}):=\begin{cases}
e_{i}x^{4l+t}	& \mbox{if $r=i$, $s=j$ and $n\equiv t$ 
(${\rm mod}\,4$) where $0\leq t\leq3$}\\ 
0		& \mbox{otherwise}
\end{cases}$$ 
and 
$$\gamma_{i,j}^{n,l}(\mathfrak{a}_{r,s}^{n}):=\begin{cases}
e_{i}y^{4l+t}	& \mbox{if $r=i$, $s=j$ and $n\equiv t$ 
(${\rm mod}\,4$) where $0\leq t\leq3$}\\ 
0		& \mbox{otherwise}
\end{cases}$$ 
for $0\leq r\leq3$ and $0\leq s\leq n$, respectively. 
\end{defin}
\noindent 
Note that, for $t\geq0$ and $0\leq i\leq3$, we get $\beta_{i,
j}^{4t,0}=\gamma_{i,j}^{4t,0}$ for $0\leq j\leq4t$ and $\beta_{i,
j}^{4t+2,T}=-\gamma_{i,j}^{4t+2,T}$ for $0\leq j\leq4t+2$.

Then, by Remark~\ref{rembimodule}, we immediately have a 
$K$-basis of ${\rm Hom}_{A_{T}^{\rm e}}(Q^{n},A_{T})$:  
\begin{lem}\label{base_hom}
Let $n\geq0$ be an integer. Then 
$$\begin{cases}
\{\beta_{i,j}^{n,0}\mid0\leq i\leq3;\ 0\leq j\leq n\}
	&\mbox{if $n\equiv0$ $({\rm mod}\,4)$ and $T=0$}\\
\{\beta_{i,j}^{n,4l},\ \gamma_{i,j}^{n,4u}\mid0\leq i\leq3;\ 
0\leq j\leq n;\ 0\leq l\leq T;\ 1\leq u\leq T\}
	&\mbox{if $n\equiv0$ $({\rm mod}\,4)$ and $T\geq1$}\\
\{\beta_{i,j}^{n,4l+1},\ \gamma_{i,j}^{n,4l+1}\mid0\leq i\leq
3;\ 0\leq j\leq n;\ 0\leq l\leq T\}
	&\mbox{if $n\equiv1$ $({\rm mod}\,4)$}\\
\{\beta_{i,j}^{n,2}\mid0\leq i\leq3;\ 0\leq j\leq n\}
	&\mbox{if $n\equiv2$ $({\rm mod}\,4)$ and $T=0$}\\
\{\beta_{i,j}^{n,4l+2},\ \gamma_{i,j}^{n,4u+2}\mid0\leq i\leq
3;\ 0\leq j\leq n;\ 0\leq l\leq T;\ 0\leq u\leq T-1\}	
	&\mbox{if $n\equiv2$ $({\rm mod}\,4)$ and $T\geq1$}\\
\{\beta_{i,j}^{n,4l+3},\ \gamma_{i,j}^{n,4l+3}\mid0\leq i\leq
3;\ 0\leq j\leq n;\ 0\leq l\leq T-1\}
	&\mbox{if $n\equiv3$ $({\rm mod}\,4)$ and $T\geq1$}
\end{cases}$$
gives a $K$-basis of ${\rm Hom}_{A^{\rm e}_{T}}(Q^{n},A_{T})$. 
Moreover, ${\rm Hom}_{A_{T}^{\rm e}}(Q^{n},A_{T})=\{0\}$, if 
$n\equiv3$ $({\rm mod}\,4)$ and $T=0$. 
\end{lem}

\noindent
In the rest of the paper, we consider the subscripts $i$ of 
all maps $\beta_{i,j}^{n,l}$ and $\gamma_{i,j}^{n,l}$ as 
modulo $4$.

\subsection{Maps ${\rm Hom}_{A_{T}^{\rm e}}(\partial^{n},A_{T})$}
Now, by direct computations, we have the images of the basis 
elements in Lemma~\ref{base_hom} under the map ${\rm Hom}_{A_{T}^{
\rm e}}(\partial^{n}, A_{T})$:

\begin{lem}\label{image_maps} 
For $m\geq0$, we have the following{\rm:} 
\begin{enumerate}[\rm(a)]
\item For $0\leq i\leq3$ and $0\leq j\leq 4m+1$, 
$$\beta_{i,j}^{4m+1,0}\partial^{4m+2}=\begin{cases}
\beta^{4m+2,T}_{i,j+1}+\beta^{4m+2,T}_{i-1,j+1}
	&\mbox{if $i=0,2$ and $0\leq j\leq2m-1$} \\
\beta^{4m+2,0}_{i,j}+\beta^{4m+2,0}_{i-1,j}
	&\mbox{if $i=1,3$ and $0\leq j\leq2m$}\\
(T+1)\beta^{4m+2,T}_{i,2m+1}+T\beta^{4m+2,T}_{i+1,2m+1} 
	& \\
\hspace{5mm}+T\beta^{4m+2,T}_{i+2,2m+1}+(T+1)\beta^{4m+2,
T}_{i+3,2m+1}
	&\mbox{if $i=0,2$ and $j=2m$,} \\
	& \quad\quad \mbox{or if $i=1,3$ and $j=2m+1$}\\
\beta^{4m+2,0}_{i,j+1}+\beta^{4m+2,0}_{i-1,j+1}
	&\mbox{if $i=0,2$ and $2m+1\leq j\leq4m+1$}\\
\beta^{4m+2,T}_{i,j}+\beta^{4m+2,T}_{i-1,j}
	&\mbox{if $i=1,3$ and $2m+2\leq j\leq4m+1$}
\end{cases}$$
and 
$$\gamma_{i,j}^{4m+1,0}\partial^{4m+2}=\begin{cases}
\gamma^{4m+2,0}_{i,j}+\gamma^{4m+2,0}_{i-1,j}
	&\mbox{if $i=0,2$ and $0\leq j\leq2m$}\\
\gamma^{4m+2,T}_{i,j+1}+\gamma^{4m+2,T}_{i-1,j+1}
	&\mbox{if $i=1,3$ and $0\leq j\leq2m-1$}\\
(T+1)\gamma^{4m+2,T}_{i,2m+1}+T\gamma^{4m+2,T}_{i+1,2m+1}
	&\\
\hspace{5mm}+T\gamma^{4m+2,T}_{i+2,2m+1}+(T+1)\gamma^{4m+2,
T}_{i+3,2m+1}
	&\mbox{if $i=0,2$ and $j=2m+1$,}\\
	&\quad\quad\mbox{or if $i=1,3$ and $j=2m$}\\
\gamma^{4m+2,T}_{i,j}+\gamma^{4m+2,T}_{i-1,j}
	&\mbox{if $i=0,2$ and $2m+2\leq j\leq4m+1$}\\
\gamma^{4m+2,0}_{i,j+1}+\gamma^{4m+2,0}_{i-1,j+1}
	&\mbox{if $i=1,3$ and $2m+1\leq j\leq4m+1$}. 
\end{cases}$$
Moreover, for $1\leq l\leq T$, $0\leq i\leq3$ and $0\leq 
j\leq 4m+1$, 
$$\beta_{i,j}^{4m+1,l}\partial^{4m+2}=\begin{cases}
0 	&\mbox{if $i=0,2$ and $0\leq j\leq2m$}\\ 
\beta^{4m+2,l}_{i,j}+\beta^{4m+2,l}_{i-1,j}
	&\mbox{if $i=1,3$ and $0\leq j\leq2m$}\\ 
\beta^{4m+2,l}_{i,j+1}+\beta^{4m+2,l}_{i-1,j+1}
	&\mbox{if $i=0,2$ and $2m+1\leq j\leq4m+1$}\\ 
0 	&\mbox{if $i=1,3$ and $2m+1\leq j\leq4m+1$}
\end{cases}$$
and 
$$\gamma_{i,j}^{4m+1,l}\partial^{4m+2}=\begin{cases}
\gamma^{4m+2,l}_{i,j}+\gamma^{4m+2,l}_{i-1,j}
	&\mbox{if $i=0,2$ and $0\leq j\leq2m$}\\ 
0	&\mbox{if $i=1,3$ and $0\leq j\leq2m$}\\ 
0	&\mbox{if $i=0,2$ and $2m+1\leq j\leq4m+1$}\\ 
\gamma^{4m+2,l}_{i,j+1}+\gamma^{4m+2,l}_{i-1,j+1}
	&\mbox{if $i=1,3$ and $2m+1\leq j\leq4m+1$.}
\end{cases}$$

\item For $0\leq l\leq T-1$ (so $T\geq1$), $0\leq i\leq3$ 
and $0\leq j\leq4m+2$, 
$$\beta^{4m+2,l}_{i,j}\partial^{4m+3}=
\begin{cases}
\beta^{4m+3,l}_{i,j}	&\mbox{if $i=0,2$ and $0\leq j\leq2m$}\\
-\beta^{4m+3,l}_{i-1,j}	&\mbox{if $i=1,3$ and $0\leq j\leq2m$}\\
\beta^{4m+3,l}_{i,2m+1}-\beta^{4m+3,l}_{i-1,2m+2}
			&\mbox{if $i=0,2$ and $j=2m+1$}\\
\beta^{4m+3,l}_{i,2m+2}-\beta^{4m+3,l}_{i-1,2m+1}
			&\mbox{if $i=1,3$ and $j=2m+1$}\\
-\beta^{4m+3,l}_{i-1,j+1}
			&\mbox{if $i=0,2$ and $2m+2\leq j\leq4m+2$}\\
\beta^{4m+3,l}_{i,j+1}	&\mbox{if $i=1,3$ and $2m+2\leq j\leq4m+2$}
\end{cases}$$
and 
$$\gamma^{4m+2,l}_{i,j}\partial^{4m+3}=
\begin{cases}
-\gamma^{4m+3,l}_{i-1,j}
		&\mbox{if $i=0,2$ and $0\leq j\leq2m$}\\
\gamma^{4m+3,l}_{i,j}
		&\mbox{if $i=1,3$ and $0\leq j\leq2m$}\\
\gamma^{4m+3,l}_{i,2m+2}-\gamma^{4m+3,l}_{i-1,2m+1}
		&\mbox{if $i=0,2$ and $j=2m+1$}\\
\gamma^{4m+3,l}_{i,2m+1}-\gamma^{4m+3,l}_{i-1,2m+2}
		&\mbox{if $i=1,3$ and $j=2m+1$}\\
\gamma^{4m+3,l}_{i,j+1}	
		&\mbox{if $i=0,2$ and $2m+2\leq j\leq4m+2$}\\
-\gamma^{4m+3,l}_{i-1,j+1}
		&\mbox{if $i=1,3$ and $2m+2\leq j\leq4m+2$.}
\end{cases}$$
Moreover, for $0\leq i\leq3$ and $0\leq j\leq 4m+2$, 
$\beta^{4m+2,T}_{i,j}\partial^{4m+3}=\gamma^{4m+2,T}_{i,j}
\partial^{4m+3}=0$. 
\item For $0\leq l\leq T-1$ (so $T\geq1$), $0\leq i\leq3$ 
and $0\leq j\leq4m+3$, 
$$\beta^{4m+3,l}_{i,j}\partial^{4m+4}=\begin{cases}
0	&\mbox{if $i=0,2$ and 
$0\leq j\leq 2m+1$}\\
\beta^{4m+4,l+1}_{i,j} + \beta^{4m+4,l+1}_{i-1,j} 
	&\mbox{if $i=1,3$ and $0\leq j\leq 2m+1$}\\
\beta^{4m+4,l+1}_{i,j+1} + \beta^{4m+4,l+1}_{i-1,j+1} 
	&\mbox{if $i=0,2$ and $2m+2\leq j\leq 4m+3$}\\
0	&\mbox{if $i=1,3$ and $2m+2\leq j\leq 4m+3$}
\end{cases}$$
and 
$$\gamma^{4m+3,l}_{i,j}\partial^{4m+4}=\begin{cases}
0	&\mbox{if $i=0,2$ and 
$0\leq j\leq 2m+1$}\\
\gamma^{4m+4,l+1}_{i,j} + \gamma^{4m+4,l+1}_{i-1,j} 
	&\mbox{if $i=1,3$ and $0\leq j\leq 2m+1$}\\
\gamma^{4m+4,l+1}_{i,j+1} + \gamma^{4m+4,l+1}_{i-1,j+1} 
	&\mbox{if $i=0,2$ and $2m+2\leq j\leq 4m+3$}\\
0	&\mbox{if $i=1,3$ and $2m+2\leq j\leq 4m+3$.}
\end{cases}$$
\item For $0\leq l\leq T$, $0\leq i\leq3$ and $0\leq j\leq4m$, 
$$\beta^{4m,0}_{i,j}\partial^{4m+1}=\gamma^{4m,0}_{i,j}
\partial^{4m+1}=\begin{cases}
\beta^{4m+1,0}_{i,j}-\gamma^{4m+1,0}_{i-1,j}+\gamma^{4m+1,T}_{i,j+1}
-\beta^{4m+1,T}_{i-1,j+1} 	& \mbox{if $i=0,2$ and $0\leq 
j\leq2m-1$}\\ 
\gamma^{4m+1,0}_{i,j}-\beta^{4m+1,0}_{i-1,j}+\beta^{4m+1,T}_{i,j+1}
-\gamma^{4m+1,T}_{i-1,j+1} 	& \mbox{if $i=1,3$ and $0\leq 
j\leq2m-1$}\\
\beta^{4m+1,0}_{i,2m}-\gamma^{4m+1,0}_{i-1,2m}+
\gamma^{4m+1,0}_{i,2m+1}-\beta^{4m+1,0}_{i-1,2m+1} 
		& \mbox{if $i=0,2$ and $j=2m$}\\
\gamma^{4m+1,0}_{i,2m}-\beta^{4m+1,0}_{i-1,2m}+
\beta^{4m+1,0}_{i,2m+1}-\gamma^{4m+1,0}_{i-1,2m+1} 
		& \mbox{if $i=1,3$ and $j=2m$}\\
\beta^{4m+1,T}_{i,j}-\gamma^{4m+1,T}_{i-1,j}+
\gamma^{4m+1,0}_{i,j+1}-\beta^{4m+1,0}_{i-1,j+1} 
		& \mbox{if $i=0,2$ and $2m+1\leq j\leq4m$}\\
\gamma^{4m+1,T}_{i,j}-\beta^{4m+1,T}_{i-1,j}+
\beta^{4m+1,0}_{i,j+1}-\gamma^{4m+1,0}_{i-1,j+1} 
		& \mbox{if $i=1,3$ and $2m+1\leq j\leq4m$,}
\end{cases}$$
and also, for $1\leq l\leq T$, $0\leq i\leq3$ and $0\leq j
\leq4m$,
$$\beta^{4m,l}_{i,j}\partial^{4m+1}=
\begin{cases}
\beta^{4m+1,l}_{i,j}	& \mbox{if $i=0,2$ and $0\leq j\leq 
2m-1$}\\
-\beta^{4m+1,l}_{i-1,j}	& \mbox{if $i=1,3$ and $0\leq j\leq 
2m-1$}\\
\beta^{4m+1,l}_{i,2m}-\beta^{4m+1,l}_{i-1,2m+1}
			& \mbox{if $i=0,2$ and $j=2m$}\\
\beta^{4m+1,l}_{i,2m+1}-\beta^{4m+1,l}_{i-1,2m}
			& \mbox{if $i=1,3$ and $j=2m$}\\
-\beta^{4m+1,l}_{i-1,j+1}& \mbox{if $i=0,2$ and $2m+1\leq j\leq 
4m$}\\
\beta^{4m+1,l}_{i,j+1}	& \mbox{if $i=1,3$ and $2m+1\leq j\leq 
4m$}
\end{cases}$$
and 
$$\gamma^{4m,l}_{i,j}\partial^{4m+1}=
\begin{cases}
-\gamma^{4m+1,l}_{i-1,j}
			& \mbox{if $i=0,2$ and $0\leq j\leq
2m-1$}\\
\gamma^{4m+1,l}_{i,j}	& \mbox{if $i=1,3$ and $0\leq j\leq
2m-1$}\\
\gamma^{4m+1,l}_{i,2m+1}-\gamma^{4m+1,l}_{i-1,2m}
			& \mbox{if $i=0,2$ and $j=2m$}\\
\gamma^{4m+1,l}_{i,2m}-\gamma^{4m+1,l}_{i-1,2m+1}
			& \mbox{if $i=1,3$ and $j=2m$}\\
\gamma^{4m+1,l}_{i,j+1}	
			& \mbox{if $i=0,2$ and $2m+1\leq j
\leq4m$}\\
-\gamma^{4m+1,l}_{i-1,j+1}	
			& \mbox{if $i=1,3$ and $2m+1\leq j
\leq4m$.}
\end{cases}$$
\end{enumerate}
\end{lem}
\noindent
The proof of this lemma follows from easy computations, and 
so we omit it.

\subsection{A basis of ${\rm Im}\,{\rm Hom}_{A_{T}^{\rm e}}
(\partial^{n}, A_{T})$}
Now, by using Lemma~\ref{image_maps}, we have a $K$-basis 
of ${\rm Im}\,{\rm Hom}_{A_{T}^{\rm e}}(\partial^{n},A_{T})$ 
for $n\geq1$: 
\begin{lem}\label{image_basis}
For $m\geq0$, we have
\begin{enumerate}[\rm(a)]
\item \begin{enumerate}[\rm(1)]
\item If $T=0$, then ${\rm Im}\,{\rm Hom}_{A_{0}^{\rm e}}
(\partial^{4m+4},A_{0})=\{0\}$. 
\item If $T>0$, then $\{\gamma_{i,j}^{4m+4,l}+
\gamma_{i-1,j}^{4m+4,l},\ \beta_{i+1,j}^{4m+4,l}+
\beta_{i,j}^{4m+4,l},\ \beta_{i,k}^{4m+4,l}+
\beta_{i-1,k}^{4m+4,l},\ \gamma_{i+1,k}^{4m+4,l}+
\gamma_{i,k}^{4m+4,l}\mid1\leq l\leq T;\ i=0,2;\ 0\leq j\leq
2m+1;\ 2m+3\leq k\leq4m+4\}$ is a $K$-basis of 
${\rm Im}\,{\rm Hom}_{A_{T}^{\rm e}}(\partial^{4m+4},A_{T})$. 
\end{enumerate}
\item \begin{enumerate}[\rm(1)]
\item If $T=0$, then $\{\gamma_{1,j}^{4m+1,0}-
\beta_{0,j}^{4m+1,0}+\beta_{1,j+1}^{4m+1,0}-
\gamma_{0,j+1}^{4m+1,0},\ \beta_{2,j}^{4m+1,0}-
\gamma_{1,j}^{4m+1,0}+\gamma_{2,j+1}^{4m+1,0}-
\beta_{1,j+1}^{4m+1,0},\ \gamma_{3,j}^{4m+1,0}-
\beta_{2,j}^{4m+1,0}+\beta_{3,j+1}^{4m+1,0}-
\gamma_{2,j+1}^{4m+1,0}\mid0\leq j\leq4m\}$ is a $K$-basis of 
${\rm Im}\,{\rm Hom}_{A_{0}^{\rm e}}(\partial^{4m+1},A_{0})$. 
\item If $T>0$, then 
$\{\gamma_{i+1,j}^{4m+1,0}-\beta_{i,j}^{4m+1,0}+
\beta_{i+1,j+1}^{4m+1,T}-\gamma_{i,j+1}^{4m+1,T},\ 
\beta_{2,j}^{4m+1,0}-\gamma_{1,j}^{4m+1,0}+
\gamma_{2,j+1}^{4m+1,T}-\beta_{1,j+1}^{4m+1,T},\ 
\gamma_{i+1,k}^{4m+1,T}-\beta_{i,k}^{4m+1,T}+
\beta_{i+1,k+1}^{4m+1,0}-\gamma_{i,k+1}^{4m+1,0},\ 
\beta_{2,k}^{4m+1,T}-\gamma_{1,k}^{4m+1,T}+
\gamma_{2,k+1}^{4m+1,0}-\beta_{1,k+1}^{4m+1,0}\mid i=0,2;\ 
0\leq j\leq2m-1;\ 2m+1\leq k\leq4m\}\cup
\{\gamma_{i+1,2m}^{4m+1,0}-\beta_{i,2m}^{4m+1,0}+
\beta_{i+1,2m+1}^{4m+1,0}-\gamma_{i,2m+1}^{4m+1,0},\ 
\beta_{2,2m}^{4m+1,0}-\gamma_{1,2m}^{4m+1,0}+
\gamma_{2,2m+1}^{4m+1,0}-\beta_{1,2m+1}^{4m+1,0}\mid i=0,2\}
\cup\{\beta_{i,j}^{4m+1,l},\ \gamma_{i+1,j}^{4m+1,l},\ 
\gamma_{i,k}^{4m+1,l},\ \beta_{i+1,k}^{4m+1,l}\mid 
1\leq l\leq T;\ i=0,2;\ 0\leq j\leq2m-1;\ 2m+2\leq k\leq4m+1\}
\cup\{\beta_{i+1,2m+1}^{4m+1,l}-\beta_{i,2m}^{4m+1,l},\ 
\gamma_{i+1,2m}^{4m+1,l}-\gamma_{i,2m+1}^{4m+1,l},\ 
\beta_{2,2m}^{4m+1,l}-\beta_{1,2m+1}^{4m+1,l},\ 
\gamma_{2,2m+1}^{4m+1,l}-\gamma_{1,2m}^{4m+1,l}\mid 
i=0,2;\ 1\leq l\leq T\}$ is a $K$-basis of 
${\rm Im}\,{\rm Hom}_{A_{T}^{\rm e}}(\partial^{4m+1},A_{T})$. 
\end{enumerate}
\item \begin{enumerate}[\rm(1)]
\item If $T=0$, then $\{\beta_{i+1,j}^{4m+2,0}+
\beta_{i,j}^{4m+2,0}\mid i=0,1,2;\ 0\leq j\leq4m+2\}$ is a 
$K$-basis of ${\rm Im}\,{\rm Hom}_{A_{0}^{\rm e}}(\partial^{4m+2},
A_{0})$. 
\item If $T>0$, then $\{\beta_{i+1,j}^{4m+2,l}+
\beta_{i,j}^{4m+2,l},\ \gamma_{i,j}^{4m+2,l}+
\gamma_{i-1,j}^{4m+2,l},\ \beta_{i,k}^{4m+2,l}+
\beta_{i-1,k}^{4m+2,l},\ \gamma_{i+1,k}^{4m+2,l}+
\gamma_{i,k}^{4m+2,l}\mid i=0,2;\ 0\leq l\leq T-1;\ 0\leq j
\leq2m;\ 2m+2\leq k\leq4m+2\}\cup\{\beta_{1,j}^{4m+2,T}+
\beta_{0,j}^{4m+2,T},\ \beta_{2,j}^{4m+2,T}+\beta_{1,j}^{4m+2,
T},\ \beta_{3,j}^{4m+2,T}+\beta_{2,j}^{4m+2,T}\mid0\leq j\leq
2m;\ 2m+2\leq j\leq 4m+2\}$
$$\cup\begin{cases}
\{\beta_{0,2m+1}^{4m+2,T}-\beta_{2,2m+1}^{4m+2,T},\ 
\beta_{1,2m+1}^{4m+2,T}-\beta_{3,2m+1}^{4m+2,T}\}
&\mbox{if ${\rm char}\, K\mid2T+1$}\\
\{\beta_{1,2m+1}^{4m+2,T}+\beta_{0,2m+1}^{4m+2,T},\ 
\beta_{2,2m+1}^{4m+2,T}+\beta_{1,2m+1}^{4m+2,T},\ 
\beta_{3,2m+1}^{4m+2,T}+\beta_{2,2m+1}^{4m+2,T}\}
&\mbox{if ${\rm char}\, K\nmid2T+1$}
\end{cases}$$
is a $K$-basis of ${\rm Im}\,{\rm Hom}_{A_{T}^{\rm e}}
(\partial^{4m+2},A_{T})$. 
\end{enumerate}
\item \begin{enumerate}[\rm(1)]
\item If $T=0$, then ${\rm Im}\,{\rm Hom}_{A_{0}^{\rm e}}
(\partial^{4m+3},A_{0})=\{0\}$. 
\item If $T>0$, then $\{\beta_{i,j}^{4m+3,l},\ \gamma_{i+1,
j}^{4m+3,l},\ \beta_{i+1,k}^{4m+3,l},\ \gamma_{i,k}^{4m+3,l}
\mid0\leq l\leq T-1;\ i=0,2;\ 0\leq j\leq2m;\ 2m+3\leq k\leq
4m+3\}\cup\{\beta_{1,2m+2}^{4m+3,l}-\beta_{0,2m+1}^{4m+3,l},\ 
\beta_{2,2m+1}^{4m+3,l}-\beta_{1,2m+2}^{4m+3,l},\ \beta_{3,
2m+2}^{4m+3,l}-\beta_{2,2m+1}^{4m+3,l},\ \gamma_{1,2m+1}^{4m+3,
l}-\gamma_{0,2m+2}^{4m+3,l},\ \gamma_{2,2m+2}^{4m+3,l}-
\gamma_{1,2m+1}^{4m+3,l},\ \gamma_{3,2m+1}^{4m+3,l}-\gamma_{2,
2m+2}^{4m+3,l}\mid0\leq l\leq T-1\}$ is a $K$-basis of 
${\rm Im}\,{\rm Hom}_{A_{T}^{\rm e}}(\partial^{4m+3},A_{T})$. 
\end{enumerate}
\end{enumerate}
\end{lem}
\noindent
The proof of this lemma follows from easy computations, and 
thus we omit it.

As an immediate consequence, we get the dimension of 
${\rm Im}\,{\rm Hom}_{A_{T}^{\rm e}}(\partial^{n},A_{T})$ for 
$n\geq1$: 
\begin{cor}\label{dim_im}
For $T\geq0$, $m\geq0$, and $0\leq r\leq3$, 
$$\dim_{K}{\rm Im}\,{\rm Hom}_{A_{T}^{\rm e}}(\partial^{4m+r},
A_{T})=\begin{cases}
16Tm 			& \mbox{if $r=0$ (where $m\neq0$)}\\
2T(8m+3)+3(4m+1)	& \mbox{if $r=1$}\\
8T(2m+1)+4(3m+2) 	& \mbox{if ${\rm char}\, K\mid2T+1$ and 
$r=2$}\\
8T(2m+1)+3(4m+3) 	& \mbox{if ${\rm char}\, K\nmid2T+1$ and 
$r=2$}\\
2T(8m+7) 		& \mbox{if $r=3$.}\\
\end{cases}$$
\end{cor}

\subsection{A basis of ${\rm Ker}\,{\rm Hom}_{A_{T}^{\rm e}}
(\partial^{n}, A_{T})$} 
By using Lemma~\ref{image_maps}, we also have a $K$-basis 
of ${\rm Ker}\,{\rm Hom}_{A_{T}^{\rm e}}(\partial^{n}, A_{T})$ 
for $n\geq1$:

\begin{lem}\label{kernel_basis}
For $m\geq0$, we have 
\begin{enumerate}[\rm(a)]
\item \begin{enumerate}[\rm(1)]

\item If $T=0$, then ${\rm Ker}\,{\rm Hom}_{A_{0}^{\rm e}}
(\partial^{4m+4}, A_{0})=\{0\}$.

\item If $T>0$, then $\{\beta_{i,j}^{4m+3,l},\ \gamma_{i+1,
j}^{4m+3,l},\ \gamma_{i,k}^{4m+3,l},\ \beta_{i+1,k}^{4m+3,l}
\mid i=0,2;\ 0\leq l\leq T-1;\ 0\leq j\leq2m+1;\ 2m+2\leq k
\leq4m+3\}$ is a $K$-basis of ${\rm Ker}\,{\rm Hom}_{A_{T}^{\rm e}}
(\partial^{4m+4},A_{T})$. 
\end{enumerate}
\item \begin{enumerate}[\rm(1)]

\item If $T=0$, then $\{\sum_{i=0}^{3}\beta_{i,j}^{4m,0}\mid0
\leq j\leq4m\}$ is a $K$-basis of ${\rm Ker}\,{\rm Hom}_{
A_{0}^{\rm e}}(\partial^{4m+1},A_{0})$.

\item If $T>0$, then $\{\sum_{i=0}^{3}\beta_{i,j}^{4m,0}\mid
0\leq j\leq4m\}\cup\{\beta_{i+1,j}^{4m,l}+\beta_{i,j}^{4m,l},\ 
\gamma_{i,j}^{4m,l}+\gamma_{i-1,j}^{4m,l},\ \gamma_{i+1,k}^{4m,
l}+\gamma_{i,k}^{4m,l},\ \beta_{i,k}^{4m,l}+\beta_{i-1,k}^{4m,
l}\mid i=0,2;\ 1\leq l\leq T;\ 0\leq j\leq2m-1;\ 2m+1\leq k
\leq4m\}\cup\{\sum_{i=0}^{3}\beta_{i,2m}^{4m,l},\ \sum_{i=0}^{3}
\gamma_{i,2m}^{4m,l}\mid1\leq l\leq T\}$ is a $K$-basis of 
${\rm Ker}\,{\rm Hom}_{A_{T}^{\rm e}}(\partial^{4m+1},A_{T})$. 
\end{enumerate}
\item \begin{enumerate}[\rm(1)]
\item If $T=0$, then $\{\beta_{0,j}^{4m+1,0}+\gamma_{0,j+1}^{4
m+1,0},\ \beta_{i+1,j+1}^{4m+1,0}-\gamma_{i,j+1}^{4m+1,0}+
\gamma_{i+1,j}^{4m+1,0}-\beta_{i,j}^{4m+1,0},\ \gamma_{2,j+1
}^{4m+1,0}-\beta_{1,j+1}^{4m+1,0}+\beta_{2,j}^{4m+1,0}-
\gamma_{1,j}^{4m+1,0}\mid i=0,2;\ 0\leq j\leq4m\}\cup\{
\gamma_{0,j}^{4m+1,0}+\beta_{1,j}^{4m+1,0}+\gamma_{2,j}^{4m+1,
0}+\beta_{3,j}^{4m+1,0}\mid0\leq j\leq2m+1\}\cup\{\beta_{0,j}^{4
m+1,0}+\gamma_{1,j}^{4m+1,0}+\beta_{2,j}^{4m+1,0}+\gamma_{3,
j}^{4m+1,0}\mid2m+1\leq j\leq4m+1\}$ is a $K$-basis of 
${\rm Ker}\,{\rm Hom}_{A_{0}^{\rm e}}(\partial^{4m+2},A_{0})$.

\item If $T>0$, then $\{\beta_{0,j}^{4m+1,0}+
\gamma_{0,j+1}^{4m+1,T},\ \gamma_{i+1,j}^{4m+1,0}-
\beta_{i,j}^{4m+1,0}+\beta_{i+1,j+1}^{4m+1,T}-\gamma_{i,
j+1}^{4m+1,T},\ \beta_{2,j}^{4m+1,0}-\gamma_{1,j}^{4m+1,0}+
\gamma_{2,j+1}^{4m+1,T}-\beta_{1,j+1}^{4m+1,T}\mid i=0,2;\ 
0\leq j\leq2m-1\}\cup\{\beta_{0,j}^{4m+1,T}+
\gamma_{0,j+1}^{4m+1,0},\ \gamma_{i+1,j}^{4m+1,T}-
\beta_{i,j}^{4m+1,T}+\beta_{i+1,j+1}^{4m+1,0}-\gamma_{i,
j+1}^{4m+1,0},\ \beta_{2,j}^{4m+1,T}-\gamma_{1,j}^{4m+1,T}+
\gamma_{2,j+1}^{4m+1,0}-\beta_{1,j+1}^{4m+1,0}
\mid i=0,2;\ 2m+1\leq j\leq4m\}\cup\{\beta_{0,2m}^{4m+1,0}+
\gamma_{0,2m+1}^{4m+1,0},\ \gamma_{i+1,2m}^{4m+1,0}-
\beta_{i,2m}^{4m+1,0}+\beta_{i+1,2m+1}^{4m+1,0}-\gamma_{i,
2m+1}^{4m+1,0},\ \beta_{2,2m}^{4m+1,0}-\gamma_{1,2m}^{4m+1,0}+
\gamma_{2,2m+1}^{4m+1,0}-\beta_{1,2m+1}^{4m+1,0}\mid i=0,2\}
\cup\{\gamma_{0,j}^{4m+1,T}+\beta_{1,j}^{4m+1,T}+
\gamma_{2,j}^{4m+1,T}+\beta_{3,j}^{4m+1,T},\  
\beta_{0,k}^{4m+1,T}+\gamma_{1,k}^{4m+1,T}+
\gamma_{2,k}^{4m+1,T}+\beta_{3,k}^{4m+1,T} 
\mid 0\leq j\leq2m;\ 2m+1\leq k\leq4m+1\}
\cup\{\beta_{i,j}^{4m+1,l},\ \gamma_{i+1,j}^{4m+1,l},\ 
\gamma_{i,k}^{4m+1,l},\ \beta_{i+1,k}^{4m+1,l}\mid
1\leq l\leq T;\ i=0,2;\ 0\leq j\leq2m;\ 2m+1\leq k\leq4m+1\}$ 
$$\cup\begin{cases}
\{\gamma_{0,2m+1}^{4m+1,0}+\gamma_{2,2m+1}^{4m+1,0},\ 
\beta_{1,2m+1}^{4m+1,0}+\beta_{3,2m+1}^{4m+1,0}\}
&\mbox{if ${\rm char}\, K\mid2T+1$}\\
\{\gamma_{0,2m+1}^{4m+1,0}+\beta_{1,2m+1}^{4m+1,0}+
\gamma_{2,2m+1}^{4m+1,0}+\beta_{3,2m+1}^{4m+1,0}
\}&\mbox{if ${\rm char}\, K\nmid2T+1$}
\end{cases}$$
is a $K$-basis of ${\rm Ker}\,{\rm Hom}_{A_{T}^{\rm e}}
(\partial^{4m+2},A_{T})$. 
\end{enumerate}
\item \begin{enumerate}[\rm(1)]
\item If $T=0$, then $\{\beta_{i,j}^{4m+2,0}\mid0\leq i\leq
3;\ 0\leq j\leq4m+2\}$ is a $K$-basis of ${\rm Ker}\,{\rm 
Hom}_{A_{0}^{\rm e}}(\partial^{4m+3},A_{0})$. 
\item If $T>0$, then $\{\beta_{i+1,j}^{4m+2,l}+\beta_{i,j}^{4m+2,
l},\ \gamma_{i,j}^{4m+2,l}+\gamma_{i-1,j}^{4m+2,l},\ \gamma_{i+1,
k}^{4m+2,l}+\gamma_{i,k}^{4m+2,l},\ \beta_{i,k}^{4m+2,l}+
\beta_{i-1,k}^{4m+2,l}\mid i=0,2;\ 0\leq l\leq T-1;\ 0\leq j
\leq2m;\ 2m+2\leq k\leq4m+2\}\cup\{\sum_{i=0}^{3}\beta_{i,
2m+1}^{4m+2,l},\ \sum_{i=0}^{3}\gamma_{i,2m+1}^{4m+2,l}\mid0
\leq l\leq T-1\}\cup\{\beta_{i,j}^{4m+2,T}\mid0\leq i\leq3;\ 
0\leq j\leq4m+2\}$ is a $K$-basis of ${\rm Ker}\,{\rm 
Hom}_{A_{T}^{\rm e}}(\partial^{4m+3},A_{T})$. 
\end{enumerate}
\end{enumerate}
\end{lem}
\noindent
The proof of this lemma follows from easy computations, 
and thus we omit it.

By the lemma above, we immediately have the dimension of 
${\rm Ker}\,{\rm Hom}_{A_{T}^{\rm e}}(\partial^{n}, A_{T})$ 
for $n\geq1$: 
\begin{cor}\label{dim_kr}
For $T\geq0$, $m\geq0$, and $0\leq r\leq3$, 
$$\dim_{K}{\rm Ker}\,{\rm Hom}_{A_{T}^{\rm e}}(\partial^{4m+r},
A_{T})=\begin{cases}
16Tm			& \mbox{if $r=0$ (where $m\neq0$)}\\
2T(8m+1)+4m+1		& \mbox{if $r=1$}\\
8T(2m+1)+4(5m+2)	& \mbox{if ${\rm char}\, K\mid2T+1$ and $r=2$}\\
8T(2m+1)+20m+7		& \mbox{if ${\rm char}\, K\nmid2T+1$ and $r=2$}\\
2T(8m+5)+4(4m+3)	& \mbox{if $r=3$.}
\end{cases}$$
\end{cor}

\subsection{The Hochschild cohomology group of $A_{T}$}
Now, by Lemmas~\ref{image_basis} and \ref{kernel_basis}, 
we have a $K$-basis of the Hochschild cohomology group 
${\rm HH}^{n}(A_{T})$ for $n\geq0$:

\begin{prop}\label{basis_HH}
For $T\geq0$ and $m\geq0$, 
\begin{enumerate}[\rm(a)]
\item If $T=0$, then 
\begin{enumerate}[\rm(1)]
\item $\{\sum_{i=0}^{3}\beta_{i,j}^{4m,0}\mid0\leq j\leq4m\}$ 
is a $K$-basis of ${\rm HH}^{4m}(A_{0})$.

\item $\{\beta_{0,j}^{4m+1,0}+\gamma_{0,j+1}^{4m+1,0}\mid0\leq 
j\leq4m\}\cup\{\gamma_{0,j}^{4m+1,0}+\beta_{1,j}^{4m+1,0}+
\gamma_{2,j}^{4m+1,0}+\beta_{3,j}^{4m+1,0}\mid 0\leq j\leq2m+1\}
\cup\{\beta_{0,j}^{4m+1,0}+\gamma_{1,j}^{4m+1,0}+\beta_{2,
j}^{4m+1,0}+\gamma_{3,j}^{4m+1,0}\mid2m+1\leq j\leq4m+1\}$ is a 
$K$-basis of ${\rm HH}^{4m+1}(A_{0})$.

\item $\{\beta_{0,j}^{4m+2,0}\mid0\leq j\leq4m+2\}$ is a 
$K$-basis of ${\rm HH}^{4m+2}(A_{0})$.

\item ${\rm HH}^{4m+3}(A_{0})=\{0\}$. 
\end{enumerate}

\item If $T>0$, then@
\begin{enumerate}[\rm(1)]
\item $\{\sum_{i=0}^{3}\beta_{i,j}^{4m,0},\ \sum_{i=0}^{3}
\beta_{i,2m}^{4m,l},\ \sum_{i=0}^{3}\gamma_{i,2m}^{4m,l}\mid
0\leq j\leq4m;\ 1\leq l\leq T\}$ is a $K$-basis of ${\rm HH}^{4m}(A_{T})$.

\item $\{\beta_{0,j}^{4m+1,0}+\gamma_{0,j+1}^{4m+1,T}\mid0
\leq j\leq2m-1\}\cup\{\beta_{0,j}^{4m+1,T}+\gamma_{0,j+1}^{4m+1,
0}\mid2m+1\leq j\leq4m\}\cup\{\beta_{0,2m}^{4m+1,0}+\gamma_{0,
2m+1}^{4m+1,0}\}\cup\{\gamma_{0,j}^{4m+1,T}+\beta_{1,j}^{4m+
1,T}+\gamma_{2,j}^{4m+1,T}+\beta_{3,j}^{4m+1,T}\mid 0\leq j\leq
2m\}\cup\{\beta_{0,j}^{4m+1,T}+\gamma_{1,j}^{4m+1,T}+\beta_{2,
j}^{4m+1,T}+\gamma_{3,j}^{4m+1,T}\mid2m+1\leq j\leq4m+1\}\cup
\{\beta_{0,2m}^{4m+1,l},\ \gamma_{0,2m+1}^{4m+1,l}\mid1\leq 
l\leq T\}$
$$\cup\begin{cases}
\{\gamma_{0,2m+1}^{4m+1,0}+\gamma_{2,2m+1}^{4m+1,0},\ 
\beta_{1,2m+1}^{4m+1,0}+\beta_{3,2m+1}^{4m+1,0}\}
&\mbox{if ${\rm char}\, K\mid2T+1$}\\
\{\gamma_{0,2m+1}^{4m+1,0}+\beta_{1,2m+1}^{4m+1,0}+
\gamma_{2,2m+1}^{4m+1,0}+\beta_{3,2m+1}^{4m+1,0}
\}&\mbox{if ${\rm char}\, K\nmid2T+1$}
\end{cases}$$
is a $K$-basis of ${\rm HH}^{4m+1}(A_{T})$. 
\item $\{\sum_{i=0}^{3}\beta_{i,2m+1}^{4m+2,l},\ 
\sum_{i=0}^{3}\gamma_{i,2m+1}^{4m+2,l}\mid0\leq l\leq T-1\}
\cup\{\beta_{0,j}^{4m+2,T}\mid0\leq j\leq2m;\ 2m+2\leq j\leq4m+2\}$
$$\cup\begin{cases}
\{\beta_{0,2m+1}^{4m+2,T},\ \beta_{1,2m+1}^{4m+2,T}\} & 
\mbox{if ${\rm char}\, K\mid2T+1$}\\
\{\beta_{0,2m+1}^{4m+2,T}\} & \mbox{if ${\rm char}\, K\nmid2T+1$}
\end{cases}$$
is a $K$-basis of ${\rm HH}^{4m+2}(A_{T})$. 
\item $\{\beta_{0,2m+1}^{4m+3,l},\ \gamma_{0,2m+2}^{4m+3,l}
\mid0\leq l\leq T-1\}$ is a $K$-basis of ${\rm HH}^{4m+3}(A_{T})$. 
\end{enumerate}
\end{enumerate}
\end{prop}

\noindent
This proposition provides us with the main result in this 
paper. 
\begin{thm}\label{dim_HH}
For $T\geq0$ and $n\geq0$, the dimension of ${\rm HH}^{n}(A_{T})$ 
is given as follows{\rm:} For $m\geq0$ and $0\leq r\leq3$, 
$$\dim_{K}{\rm HH}^{4m+r}(A_{T})=
\begin{cases}
2T+4m+1		& \mbox{if $r=0$}\\
2T+8m+5		& \mbox{if ${\rm char}\, K\mid2T+1$ and $r=1$}\\
2T+4(2m+1)	& \mbox{if ${\rm char}\, K\nmid2T+1$ and $r=1$}\\
2T+4(m+1)	& \mbox{if ${\rm char}\, K\mid2T+1$ and $r=2$}\\
2T+4m+3		& \mbox{if ${\rm char}\, K\nmid2T+1$ and $r=2$}\\
2T		& \mbox{if $r=3$.}
\end{cases}$$
\end{thm}

\begin{remark}
There is an isomorphism ${\rm HH}^{0}(A_{T})\rightarrow Z(A_{T})$; 
$\phi\mapsto\phi(\sum_{i=0}^{3}e_{i}\otimes e_{i})$ of 
algebras, where $Z(A_{T})$ is the centre of $A_{T}$. Hence, 
by Proposition~\ref{basis_HH}~(a)(1), $Z(A_{T})\simeq K[x,y]
/\langle x^{T+1},xy,y^{T+1}\rangle$. 
\end{remark}

\section{Hochschild cohomology ring modulo nilpotence of 
$A_{T}$ for $T=0$}\label{modulo_nilpotence}

Throughout this section, we keep the notation from the previous 
sections, and suppose that $T=0$, namely, we only deal with 
the Koszul self-injective algebra $A_{0}$. For simplicity, we 
denote the algebra $A_{0}$ by $A$.

Recall that the Hochschild cohomology ring of $A$ is defined 
as the graded ring ${\rm HH}^{*}(A):=\bigoplus_{t\geq0}
{\rm HH}^{t}(A)=\bigoplus_{t\geq0}{\rm Ext}_{A^{\rm e}}^{t}
(A,A)$ with the Yoneda product. Let $\mathcal{N}_{A}$ be the 
ideal generated by all homogeneous nilpotent elements in 
${\rm HH}^{*}(A)$. Note that $\mathcal{N}_{A}$ is a homogeneous 
ideal in ${\rm HH}^{*}(A)$. The purpose in this section is to 
find generators and relations of the Hochschild cohomology 
ring modulo nilpotence, ${\rm HH}^{*}(A)/\mathcal{N}_{A}$, 
of $A$. It is known that ${\rm HH}^{*}(A)/\mathcal{N}_{A}$ 
is a commutative graded algebra. We denote by ${\rm HH}^{4*}(A)$
the graded subalgebra $\bigoplus_{i\geq0}{\rm HH}^{4i}(A)$ 
of ${\rm HH}^{*}(A)$, and by $\times$ the Yoneda product in 
${\rm HH}^{*}(A)$.

\begin{thm}\label{ring_modulo_nilpotence}
There is the following isomorphism of commutative graded 
algebras{\rm:} 
\begin{align*}
&{\rm HH}^{*}(A)/\mathcal{N}_{A}\simeq{\rm HH}^{4*}(A)\\
&\simeq K[z_{0},z_{1},z_{2},z_{3},z_{4}]/\langle z_{0}z_{2}-
z_{1}^{2},\ z_{0}z_{3}-z_{1}z_{2},\ z_{0}z_{4}-z_{2}^{2},\ 
z_{0}z_{4}-z_{1}z_{3},\ z_{1}z_{4}-z_{2}z_{3},\ z_{2}z_{4}-
z_{3}^{2}\rangle
\end{align*}
with $z_{j}$ in degree $4$ for $0\leq j\leq4$. Therefore 
${\rm HH}^{*}(A)/\mathcal{N}_{A}$ is finitely generated 
as an algebra. 
\end{thm}

\begin{proof}
First we establish the second isomorphism. For $0\leq j\leq4$ 
let $z_{j}:=\sum_{i=0}^{3}\beta_{i,j}^{4,0}\in{\rm HH}^{4}(A)$, 
and for $k\geq0$ define the homomorphism of $A$-$A$-bimodules 
$\sigma^{k}_{j}: Q^{k+4}\rightarrow Q^{k}$ by 
$$\mathfrak{a}_{r,s}^{k+4}\mapsto\begin{cases}
\mathfrak{a}_{r,t}^{k} 	& \mbox{if $s=j+t$ for some integer 
$t$ with $0\leq t\leq k$}\\
0			& \mbox{otherwise.}
\end{cases}$$
Then $\sigma^{k}_{j}$ are liftings of $z_{j}$, namely, 
$z_{j}=\sigma_{j}^{0}\partial^{0}$ and $\sigma_{j}^{l}
\partial^{l+5}=\partial^{l+1}\sigma_{j}^{l+1}$ hold for 
$l\geq0$. Thus, for integers $0\leq u,v\leq4$, it follows 
that the composite $z_{v}\sigma^{4}_{u}: Q^{8}\rightarrow A$ 
is given by 
$$\mathfrak{a}_{r,s}^{8}\mapsto
\begin{cases}
e_{r}	& \mbox{if $s=u+v$}\\
0	& \mbox{otherwise.}
\end{cases}$$
Note that the product $z_{v}\times z_{u}\in {\rm HH}^{8}(A)$ 
is represented by this map.

Now, for each positive integer $t$, let $i_{t}$ be an integer 
with $0\leq i_{t}\leq4$. Then, for any $w\geq2$, it can be 
shown by induction on $w$ that the product $z_{i_{1}}\times 
\cdots\times z_{i_{w}}$ is represented by the map 
$$Q^{4w}\rightarrow A;\ \mathfrak{a}_{r,s}^{4w}
\mapsto\begin{cases}
e_{r}	& \mbox{if $s=\sum_{p=1}^{w}i_{p}$}\\
0	& \mbox{otherwise.}
\end{cases}$$ 
This tells us that ${\rm HH}^{4*}(A)$ is generated by 
$z_{0},\ldots,z_{4}\in{\rm HH}^{4}(A)$. Also, for two products 
$z_{j_{1}}\times\cdots\times z_{j_{u}}$ and $z_{k_{1}}\times
\cdots\times z_{k_{u}}$ in ${\rm HH}^{4u}(A)$ (where $u\geq1$ 
and $0\leq j_{p}, k_{p}\leq4$ for $1\leq p\leq u$), we have 
that $z_{j_{1}}\times\cdots\times z_{j_{u}}=z_{k_{1}}\times
\cdots\times 
z_{k_{u}}$ if and only if $\sum_{p=1}^{u}j_{p}=\sum_{p=1}^{u}
k_{p}$. This gives the relation $z_{i}z_{j}-z_{k}z_{l}$ for 
$0\leq i$, $j$, $k$, $l\leq4$ such that $i+j=k+l$, and then, 
considering all possible elements in ${\rm HH}^{4*}(A)$, we 
get the second isomorphism.

Now, using the second isomorphism, we easily see that all 
elements in ${\rm HH}^{4*}(A)$ are not nilpotent. Furthermore, 
for $m\geq0$ and $l=1,2,3$, the images of all basis elements 
of ${\rm HH}^{4m+l}(A)$ of Proposition~\ref{basis_HH}~(a) are 
in $\mathfrak{r}_{A}$, so that, by \cite[Proposition~4.4]{SS2}, 
${\rm HH}^{4m+l}(A)$ is contained in $\mathcal{N}_{A}$. Hence 
we have the first isomorphism. Therefore the proof of the 
theorem is completed. 
\end{proof}

\noindent
Finally we consider the graded centre of the Ext algebra 
$E(A):=\bigoplus_{i\geq0}{\rm Ext}^{i}_{A}(A/\mathfrak{r}_{A},
A/\mathfrak{r}_{A})$ of $A$. Recall that the graded centre 
$Z_{\rm gr}(E(A))$ of $E(A)$ is the subring of $E(A)$ 
generated by all homogeneous elements $x$ in $E(A)$ such that 
$xy=(-1)^{|x|\cdot|y|}yx$ for all homogeneous elements $y\in 
E(A)$, where $|x|$ and $|y|$ denote the degree of $x$ and $y$, 
respectively. Let $\mathcal{N}'_{A}$ be the ideal of $Z_{\rm gr}
(E(A))$ generated by all homogeneous nilpotent elements. Since 
$A$ is a Koszul algebra, we know from \cite{BGSS} that 
$Z_{\rm gr}(E(A))/\mathcal{N}'_{A}\simeq{\rm HH}^{*}(A)/
\mathcal{N}_{A}$ as graded rings. Therefore, by 
Theorem~\ref{ring_modulo_nilpotence}, we have the following:

\begin{cor} 
There is the following isomorphism of commutative graded 
rings: 
\begin{align*}
&Z_{\rm gr}(E(A))/\mathcal{N}'_{A}\simeq{\rm HH}^{4*}(A)\\
&\simeq K[z_{0},z_{1},
z_{2},z_{3},z_{4}]/\langle z_{0}z_{2}-z_{1}^{2},\ z_{0}
z_{3}-z_{1}z_{2},\ z_{0}z_{4}-z_{2}^{2},\ z_{0}z_{4}-z_{1}
z_{3},\ z_{1}z_{4}-z_{2}z_{3},\ z_{2}z_{4}-z_{3}^{2}
\rangle
\end{align*}
with $z_{j}$ in degree $4$ for $0\leq j\leq4$. 
\end{cor}

\begin{remark}
We notice that our algebra $A_{T}$ $(T\geq0)$ belongs to the 
class of more general algebras $B_{k,s}$ defined as follows: 
For $s\geq1$, let $\Delta_{s}$ be the quiver 
$$
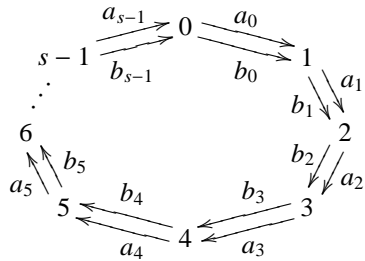
\begin{xy}
(0,14)		*+{0}="e0",
(16,10)		*+{1}="e1",
(21,0)		*+{2}="e2",
(16,-10)	*+{3}="e3",
(0,-14)		*+{4}="e4",
(-16,-10)	*+{5}="e5",
(-21,0)		*+{6}="e6",
(-16,10)	*+{s-1}="es-1",
(8,15)		*{a_{0}},
(8,8)		*{b_{0}},
(22,6.5)	*{a_{1}},
(15.5,3.5)	*{b_{1}},
(22,-6.5)	*{a_{2}},
(15.5,-3)	*{b_{2}},
(9,-15.5)	*{a_{3}},
(9,-8)		*{b_{3}},
(-8,15.5)	*{a_{s-1}},
(-7,8)		*{b_{s-1}},
(-21.5,-7.5)	*{a_{5}},
(-14.5,-4)	*{b_{5}},
(-7,-15.75)	*{a_{4}},
(-7,-8.25)	*{b_{4}},
(-20,3)		*{\cdot},
(-19,4.5)	*{\cdot},
(-18,6)		*{\cdot},
\ar @<1mm> "e0"; "e1",
\ar @<-1mm> "e0"; "e1",
\ar @<1mm> "e1"; "e2",
\ar @<-1mm> "e1"; "e2",
\ar @<1mm> "e2"; "e3",
\ar @<-1mm> "e2"; "e3",
\ar @<1mm> "e3"; "e4",
\ar @<-1mm> "e3"; "e4",
\ar @<1mm> "e4"; "e5",
\ar @<-1mm> "e4"; "e5",
\ar @<1mm> "e5"; "e6",
\ar @<-1mm> "e5"; "e6",
\ar @<1mm> "es-1"; "e0",
\ar @<-1mm> "es-1"; "e0"
\end{xy}$$
and for $k\geq2$ let $J_{k}:=\langle xy,x^{k}+y^{k},yx\rangle
\subseteq K\Delta_{s}$, where we put $x:=\sum_{i=0}^{s-1}a_{i}
\in K\Delta_{s}$ and $y:=\sum_{i=0}^{s-1}b_{i}\in K\Delta_{s}$. 
Define the algebra $B_{k,s}:=K\Delta_{s}/J_{k}$. Then $B_{k,s}$ 
is a self-injective algebra.

The results in this paper provide the computations of the 
Hochschild cohomology groups of $B_{k,s}$ for $s=4$ and 
$k=4T+2$ $(T\geq0)$. On the other hand, if $s=2$ and $k=2m$ 
$(m\geq1)$, then we know from \cite{ST} generators and 
relations of the Hochschild cohomology ring of $B_{2m,2}$. 
(In fact $B_{2m,2}$ is isomorphic to the algebra $\Lambda_{N}$ 
discussed in \cite{ST}, where $N=m$ and the quiver of 
$\Lambda_{N}$ has $2$ vertices.) For the other cases, the 
computations of the Hochschild cohomologies seem to be 
unknown. 
\end{remark}

\vspace{1mm}       
\noindent
{\bf Acknowledgement}\\[3mm]
The author would like to thank the referee for valuable 
suggestions and helpful comments. 





\bibliographystyle{model1a-num-names}
\bibliography{<your-bib-database>}



\end{document}